\documentclass[a4paper,10pt]{article}
\pdfoutput=1

\usepackage[utf8]{inputenc}

\usepackage{amsmath,amsfonts,amsthm,amssymb}
\usepackage{graphicx}
\usepackage[colorlinks,citecolor=blue,linkcolor=blue]{hyperref}
\usepackage[nameinlink,noabbrev,capitalize]{cleveref}
\usepackage{tabularx}
\numberwithin{equation}{section}

\usepackage[
	style=numeric,
 	url=false,
	doi=true,                
	eprint=true,
 	giveninits=true,         
	maxbibnames=99,          
	maxcitenames=2,
	backend=biber,           
	safeinputenc,            
]{biblatex}
\newtheorem{theorem}{Theorem}[section]
\newtheorem{lemma}[theorem]{Lemma}
\newtheorem{definition}[theorem]{Definition}

\newtheorem{remark}[theorem]{Remark}
\newtheorem{corollary}[theorem]{Corollary}

\newtheorem{algorithm}{Algorithm}
\newtheorem{assumption}{Assumption}

\newcommand\be{\begin{equation}}
\newcommand\ee{\end{equation}}

\renewcommand{\subset}{\subseteq}

\newcommand{\dx}{\,\text{\rm{}d}x}

\def\N{\mathbb  N}
\def\R{\mathbb  R}

\def\meas{\operatorname{meas}}

\def\supp{\operatorname{supp}}

\def\argmin{\operatorname*{arg\, min}}
\def\argmax{\operatorname*{arg\, max}}

\renewcommand{\phi}{\varphi}
\newcommand{\Utau}{{\mathcal U_\tau}}

\newcolumntype{L}{>{$}l<{$\quad}}
\newcolumntype{R}{>{$}r<{$\quad}}
\newcolumntype{C}{>{$}c<{$}}

\title{Optimal control problems with $L^0(\Omega)$ constraints: maximum principle and proximal gradient method}

\author{Daniel Wachsmuth%
\footnote{Institut f\"ur Mathematik,
Universit\"at W\"urzburg,
97074 W\"urzburg, Germany, {\tt daniel.wachsmuth@mathematik.uni-wuerzburg.de}.
This research was partially supported by the German Research Foundation DFG under project grant Wa 3626/3-2.}}

\begin{document}

\maketitle

\paragraph{Abstract.}
We investigate optimal control problems with $L^0$ constraints, which restrict the measure of the support of the controls.
We prove necessary optimality conditions of Pontryagin maximum principle type.
Here, a special control perturbation is used that respects the $L^0$ constraint.
First, the maximum principle is obtained in integral form, which is then turned into a pointwise form.
In addition, an optimization algorithm of proximal gradient type is analyzed. Under some assumptions,
the sequence of iterates contains strongly converging subsequences, whose limits are feasible and satisfy
a subset of the necessary optimality conditions.


\section{Introduction}

We are interested in the following optimal control problem
written as an optimation problem:
\be\label{eq001}
 \min_{u\in L^2(\Omega)} f(u) + \frac\alpha2 \|u\|_{L^2(\Omega)}^2
\ee
subject to
\be\label{eq003}
\|u\|_0 \le \tau.
\ee
Here, $\Omega\subset \R^d$ is an open set supplied with the Lebesgue measure, $f:L^2(\Omega)\to \R$ abstracts the state equation and smooth ingredients of the control problem,
$\alpha\ge0$ is a parameter.
The constraint \eqref{eq003} uses the so-called $L^0$ norm (which is -- of course -- not a norm) that is defined
for measurable $u:\Omega \to \R$ by
\[
 \|u\|_0 :=\meas \{ x: \ u(x)\ne 0\}.
\]
Of course, $\tau \in (0, \meas(\Omega))$ is a meaningful restriction.

The motivation to study such problems comes from sparse control: Find a control with small support, in our case: with prescribed size of support.
The main challenge is the discontinuity and non-convexity of the $\|\cdot\|_0$-functional: Methods from differentiable or convex optimization are not applicable.
In addition, due to the lack of weak lower continuity it is not possible to ensure existence of solutions in spaces of integrable functions.
Nevertheless, we can look into optimality conditions that need to be satisfied at a solution.
In order to study necessary optimality conditions, we will employ the Pontryagin maximum principle, which is first obtained in integral form, and then
turned into a pointwise condition by means of natural arguments.

Let us mention related works. Optimal control problems with $L^0$ norm of the control in the cost function were investigated in \cite{ItoKunisch2014,DWachsmuth2019}.
An actuator design problem is studied in  \cite{KaliseKunischSturm2018}: the controlled source term in the equation is $\chi_\omega u$, where
$\chi_\omega$ is the characteristic function of $\omega$, and the subset $\omega$ and the control $u$ are optimization variables.
An additional volume constraint is posed on $\omega$, which is equivalent to a $L^0$ constraint on $\chi_\omega u$.
In that work, shape calculus and topological derivatives with respect to $\omega$ are studied. Unfortunately, no optimality conditions
involving these topological derivatives are given, which could be compared to our results. This is subject to future work.
In the recent work \cite{ButtazzoMaialeVelichkov2021}, a shape optimization problem is turned into a problem with $L^0$ constraints.
There the control problem is posed in $W^{1,p}$, and offers  different challenges than the setting considered here.
That work will become relevant if one wants to study the regularization of \eqref{eq001}--\eqref{eq003} in $W^{1,p}$ spaces,
which would  guarantee existence of solutions due to the compact embedding of $W^{1,p}$ in $L^p$.

In this article, we will prove optimality conditions of Pontryagin maximum principle type.
Related works can be found, e.g., in \cite{CasasRaymondZidani1998,Casas1994,RaymondZidani1998}.
Those results are not directly applicable in our situation, since they do not cover $L^0$ constraints.
We will use a modification
of the control perturbations considered in \cite{CasasRaymondZidani1998,Casas1994,RaymondZidani1998}
that is adapted to the $L^0$ constraints.
These will give the maximum principle in integral form, see \cref{thm_pmp_integrated}. In order to turn it into
pointwise conditions in \cref{thm_pmp_pointwise}, we study integral minimization problems in Section \ref{sec3}.

In \cref{sec5}, we investigate an proximal gradient type algorithm,
which extends our earlier works \cite{NatemeyerWachsmuth2021,DWachsmuth2019},
where optimization problems with  $L^0$ and $L^p$, $p\in (0,1)$, functionals were considered.
Due to the simple nature of its sub-problems, this method is easy to implement.
Other methods in finite-dimensional $L^0$ constrained (or cardinality constrained) optimization include
augmented Lagrangian methods \cite{KanzowRaharjaSchwartz2021} and DC-based reformulations \cite{GotohTakedaTono2018}.
We will prove some convergence results for the proximal gradient method.
As it turns out, limit points of iterates do not satisfy the necessary condition \cref{thm_noc_sec5} but only
a subset of those, see \cref{thm_strong_conv_iterates}.
We hope that this work initiates further research on algorithms with $L^0$ constraints in an infinite-dimensional setting.

\paragraph{Notation}

We will frequently use the following notation:
For a measurable set $A$, we denote its characteristic function by $\chi_A$.
The integrand in the $L^0$ norm is abbreviated by
\[
	|u|_0:=\begin{cases} 1 & \text{ if } u\ne0, \\ 0 & \text{ if } u=0. \end{cases}
\]
Then $\|u\|_0 = \int_\Omega |u(x)|_0 \dx$. Note that $u\mapsto |u|_0$ is neither continuous nor convex but lower semicontinuous.
In addition, $u\mapsto \|u\|_0$ as mapping from $L^p(\Omega)$ to $\R$ is lower semicontintinuous but not weakly lower semicontinuous.
Moreover, we will denote the support of the measurable function $u$ by
\[
 \supp u:=\{ x\in \Omega: \ u(x)\ne0\}.
\]

\section{Maximum principle for control of ordinary differential equations}
\label{sec2}

Let us briefly and formally derive the maximum principle for an optimal control problem subject to ordinary differential equations with constraint $\|u\|_0 \le \tau$,
which serves as benchmark for  more general situations.
For illustration, let us consider the following control problem in Mayer form:
Minimize
\[
 l(x(T))
\]
subject to
\[
 x'(t) = f(t,x(t),u(t)) \text{ a.e.\ on } (0,T),
\]
\[
 x(0) = x_0,
\]
\[
 u(t) \in U \text{ for almost all } t\in (0,T),
\]
and
\[
 \|u\|_0 \le \tau.
\]
Here, $T>0$ is fixed, and $x:(0,T)\to \R^n$ and $u:(0,T)\to \R$ are the state and control.
The functions $f:\R \times \R^n \times \R$ and $l:\R^n\to \R$ are assumed to be smooth for simplicity.
Employing a standard procedure,
the constraint $ \|u\|_0 \le \tau$ can be written equivalently as an additional end-point constraint on the artificial state
$x_{n+1}$ as follows:
\[
 x_{n+1}(0) = 0, \ x_{n+1}'(t) = |u(t)|_0 \text{ a.e.\ on } (0,T), \ x_{n+1}(T) \le \tau.
\]
Let us set $\tilde f(t,x,u):=(f(t,x,u),  |u|_0)$.
Then the classical maximum principle for an optimal control $\bar u$ with state $(\bar x,\bar x_{n+1})$ and adjoint $(\bar p,\bar p_{n+1})\in\R^{n+1}$ is:
there are $(\lambda_0,\lambda_{n+1})\ne 0$, $\lambda_0\ge0$, such that the following conditions are satisfied:
\[
 \bar u(t) = \argmax_{u\in U} H(t, \bar x(t), u, \bar p(t)) + \bar p_{n+1}(t) |u|_0 ,
\]
where $H(t,x,u,p):=p^Tf(t,x,u)$
is the Hamiltonian of the original problem, and $\bar p$ solves the adjoint system
\[
 -\bar p(T)=\lambda_0 l'(\bar x(T)), \  -\bar p'(t) = f_x(t,\bar x(t),\bar u(t))^T\bar p(t) \text{ a.e.\ on } (0,T)
\]
and
\[
 -\bar p_{n+1}(T) = \lambda_{n+1}, \ -\bar p_{n+1}'(t) = 0,
\]
\[
 \lambda_{n+1}\ge 0, \ \lambda_{n+1}( \bar x_{n+1}(T) - \tau) = 0.
\]
Hence, $\bar p_{n+1}$ is constant, $\bar p_{n+1}\le0$, and $-\bar p_{n+1}$ can be interpreted as Lagrange multiplier to the constraint $\|u\|_0\le \tau$.
For a precise formulation of the maximum principle, we refer to \cite{IoffeTihomirov1979}.
In order to obtain the system in qualified form, i.e., $\lambda_0>0$, one needs additional conditions (constraint qualifications).

Summarizing the above considerations,
the following two conditions will serve as necessary optimality conditions: $\bar u(t)$ maximizes the penalized Hamiltonian, i.e.,
\begin{equation}\label{eq_ode_pmp1}
 \bar u(t) = \argmax_{u\in U} H(t, \bar x(t), u, \bar p(t)) - \lambda |u|_0 ,
\end{equation}
and $\lambda\ge0$ satisfies the complementarity condition
\begin{equation}\label{eq_ode_pmp2}
	\lambda ( \|\bar u\|_0 - \tau) = 0.
\end{equation}

Let us now transfer these results to optimal control problems, where the control is no longer defined on a subset of the real line but defined on
a set $\Omega\subset\R^d$, $d>1$.
For illustration,
let now $\Omega\subset\R^d $, $d>1$, be a bounded domain.
As above, we want to
translate the control constraint $\|u\|_0\le \tau$ to an auxiliary state constraint.
In fact, if we define $y_0$ as the weak solution in $H^1(\Omega)$ of the auxiliary state equation
\[
 -\Delta y + y = |u|_0, \ \frac{\partial y}{\partial n}=0
\]
then it follows $\int_\Omega y_0\dx  = \int_\Omega |u|_0 \dx$,
and the control constraint $\|u\|_0\le \tau$ is equivalent to the constraint $\int_\Omega y_0\dx \le \tau$ on the auxiliary state $y$.
In \cite{Casas1994}, the maximum principle for problems with elliptic partial differential equations was obtained.
In order to get a system in qualified form ($\lambda_0>0$) strong stability is used: the optimal value function has to be locally Lipschitz continuous with respect to the parameter $\tau$.
To the best of our knowledge,
such a result is not available
in the literature
for the $L^0$ constraints considered here.

Thus, we will proceed differently. We will not formulate the integral control constraint as a state constraint. Rather we will modify the technique of \cite{Casas1994}
to only consider perturbations that satisfy the constraint. In this way, we get a maximum principle in integral form satisfied for all functions $v$ with $\|v\|_0\le \tau$.
This integral maximum principle can be translated into a pointwise one.
In that way, we get the final system in qualified form while circumventing the strong stability requirement.

\section{Optimality conditions for integral functionals}
\label{sec3}

First, we are going to derive optimality conditions for integral functionals.
This is later used to transform the maximum principle from integral to pointwise form.
We will consider integral functionals generated by normal integrands.
In this section, let $\Omega\subset \R^d$ be a Lebesgue measurable set.

\begin{definition}
 The function $f:\Omega \times \R\to\R\cup\{+\infty\}$ is called {\em normal integrand}  if there exist Caratheodory functions $(f_n)_{n\in \N}$ such that
 for all $u$ and almost all $x\in \Omega$
 \[
  f(x,u) = \sup_n f_n(x,u).
 \]
 for all $u$ and almost all $x\in \Omega$.
\end{definition}

This definition is from \cite[Def.\@ 1]{BerliocchiLasry73} with equivalent characterizations in \cite[Thm.\@ 2]{BerliocchiLasry73}.
%
We will now prove optimality conditions for the following problem:
Minimize
\begin{subequations} \label{eq_int_prob}
\begin{equation}
\int_\Omega g(x,u(x))\dx
\end{equation}
subject to the constraint
\begin{equation}\label{eq_l0_constr}
 u \in \Utau:= \{ u \text{ measurable}: \ \|u\|_0\le \tau\}.
\end{equation}
\end{subequations}

Here, the minimization is over all measurable $u$ such that $g(\cdot,u)$ is integrable.
Clearly, if  $\bar u$ is a solution of  \eqref{eq_int_prob}  and
$\bar u(x)\ne0$ then $g(x,\bar u(x)) = \inf_{v\in \R} g(x,v)$. The latter function will play an important role in the subsequent analysis.
We work with the following assumption.

\begin{assumption}\label{ass_int_g}
\phantom{\quad}
\begin{enumerate}
 \item $g:\Omega \times \R \to \R \cup\{+\infty\}$ is a normal integrand,
 \item
$g(\cdot,0)$ is integrable,
 \item $\tau \in (0,\meas(\Omega))$.
\end{enumerate}
\end{assumption}

\noindent
Let us define the non-positive function $\tilde v:\Omega \to \R \cup \{-\infty\}$
 by
	\begin{equation}\label{eq_def_tildev}
	\tilde v(x) := \inf_{v\in \R} g(x,v)-g(x,0).
\end{equation}
We start with a technical lemma that helps to prove integrability of $\tilde v$ under suitable assumptions.

\begin{lemma}\label{lem_un_vn}
 Let $g:\Omega \times \R \to \R \cup\{+\infty\}$ be a normal integrand.
 Then there are measurable functions $v_n$ and $u_n$ such that
 \[
  v_n(x) = \inf_{|u|\le n} g(x,u)-g(x,0) = g(x,u_n(x))-g(x,0), \ |u_n(x) |\le n
 \]
 for almost all $x\in \Omega$.
 In addition, $\tilde v$ is measurable.
\end{lemma}
\begin{proof}
Let $(g_n)$ be Caratheodory functions such that $g(x,u)=\sup_n g_n(x,u)$
 for all $u$ and almost all $x\in \Omega$.
Let us define the set-valued mapping
\[
 E(x) := \{ (v,t): \ g(x,v)-g(x,0) \le t\},
\]
so $E(x)$ is the epi-graph of $v\mapsto g(x,v)-g(x,0)$. Then it holds
\[
 E(x) = \bigcap_n \{ (v,t): \ g_n(x,v)-t \le g(x,0)\}.
\]
Each of the set-valued mappings in the intersection is measurable by \cite[Thm. 8.2.9]{AubinFrankowska1990}, so $E$ is measurable by \cite[Thm. 8.2.4]{AubinFrankowska1990}.
Using \cite[Thm. 8.2.11]{AubinFrankowska1990}, we get the measurability of $u_n$ and $v_n$.
Measurability of $\tilde v$ is a consequence of $\tilde v(x) = \inf_n v_n(x)$.
\end{proof}

Using the function $\tilde v$ from \eqref{eq_def_tildev}, we define the sets
\begin{equation}\label{eq_def_omegas}
    \Omega_{< s}:=\{x: \ \tilde v(x)<s\}, \ \Omega_{\le s}:=\{x: \ \tilde v(x)\le s\}.
\end{equation}

\begin{lemma}\label{lem_basic_ineq}

 Let $u\in \Utau$ be given such that $g(\cdot,u)$ is integrable.
 Let  $s\le0$ and $S\subset \Omega$ with $\meas(S)=\tau$ be such that
 $\tilde v$ (see \eqref{eq_def_tildev}) is integrable on $S$ and
\[
	 \Omega_{< s} \subset S\subset \Omega_{\le s}.
\]
Then it holds
\[
 \int_\Omega  g(x,u(x)) - g(x,0)\dx \ge \int_S \tilde v(x) \dx .
\]
This inequality is satisfied with equality only if the following conditions are satisfied:
\begin{enumerate}
\item $g(x,u(x)) - g(x,0) = \tilde v(x)$ for almost all $x\in \supp u$,
 \item $\Omega_{< s} \subset \supp u\subset \Omega_{\le s}$,
 \item $s=0$ or $\meas(\supp u)=\tau$.
\end{enumerate}
\end{lemma}
\begin{proof}

Let $A:=\supp u$. Then $\meas(A) \le \tau = \meas(S)$, and it follows $\meas(A \setminus S) \le \meas(S \setminus A)$.
Using \eqref{eq_def_omegas}, we estimate
	\[\begin{split}
		\int_\Omega  g(x, u(x)) - g(x,0)\dx
		&\ge \int_A \tilde v(x) \dx    \\
		& = \int_{A \cap S} \tilde v(x) \dx + \int_{A\setminus S} \tilde v(x) \dx \\
		& \ge \int_{A \cap S} \tilde v(x) \dx + s \meas(A \setminus S) \\
		& \ge \int_{A \cap S} \tilde v(x) \dx + s \meas(S \setminus A) \\
		& \ge \int_{S} \tilde v(x) \dx .
		\end{split}
	\]
Equality in the above chain of inequalities is obtained only
if (a) $g(\cdot,u) - g(\cdot,0) = \tilde v$ on $A$, (b) $\tilde v = s$ on $A\setminus S$, hence $A  \subset \Omega_{\le s}$,
(c) $s (\meas(A \setminus S) - \meas(S \setminus A))=0$, and
(d) $s \meas(S \setminus A) = \int_{S \setminus A} \tilde v\dx$.
Condition (d) implies $\tilde v = s$ on $S \setminus A$, hence $\Omega_{< s} \subset A$.
If $s\ne0$ then condition (c) implies $\meas(A)=\tau$.
\end{proof}

With the help of these sets, we can fully characterize the solutions of \eqref{eq_int_prob}.

\begin{theorem}\label{lem_char_sol_tau}
Let \cref{ass_int_g} be satisfied.
Then $\bar u$ is a solution of \eqref{eq_int_prob}
if and only if
there are $s\le0$ and $A\subset \Omega$ with $\meas(A)=\tau$ such that
\begin{equation}\label{eq_optcon_support}
	 \Omega_{< s} \subset \supp \bar u \subset A\subset \Omega_{\le s},
\end{equation}
$\tilde v$ is integrable on $A$,
and
\begin{equation}\label{eq_optcon_u_tildev}
	g(x,\bar u(x)) - g(x,0) = \tilde v(x) \text{ for almost all  }  x\in A,
\end{equation}
where $\tilde v$ is defined in \eqref{eq_def_tildev}.
\end{theorem}
\begin{proof}
Let $\bar u$ be a solution of \eqref{eq_int_prob}.
Let $(v_n)$ and $(u_n)$ be given by \cref{lem_un_vn}.
By construction, $(v_n(x))$ is monotonically decreasing and $v_n(x)\to\tilde v(x)$ for almost all $x\in \Omega$.
Let $B\subset \Omega$ with $\meas(B)\le \tau$.
We want to show that $\chi_Bu_n$ is feasible for \eqref{eq_int_prob}.
It remains to argue that $g(\cdot,\chi_Bu_n)$ is integrable.
If the negative part of $g(\cdot,\chi_Bu_n)$ would not be integrable, then problem \eqref{eq_int_prob} would be unsolvable,
as we could find subsets $B_k\subset B$ such that $\int_\Omega g(x,\chi_{B_k} u_n)\dx \to -\infty$ for $k\to\infty$.
So the negative part of $g(\cdot,\chi_Bu_n)$ is integrable,
and the integrability of $g(\cdot,\chi_Bu_n)$ is a consequence of $g(x,\chi_B u_n(x)) - g(x,0) \le 0$ for almost all $x$.

Then $\chi_B u_n$ is feasible for \eqref{eq_int_prob},
which implies
\[
 0 \ge \int_B v_n\dx = \int_\Omega g(x,\chi_B u_n(x)) - g(x,0)\dx
 \ge \int_\Omega g(x,\bar u(x)) - g(x,0)\dx.
\]
By the monotone convergence theorem, it follows that $\tilde v$ is integrable on $B$ and
\begin{equation}\label{eq_tildev_lower}
\int_B \tilde v \dx \ge \int_\Omega g(x,\bar u(x)) - g(x,0)\dx.
\end{equation}

	The increasing functions $s\mapsto \meas(\Omega_{<s})$ and $s\mapsto \meas(\Omega_{\le s})$
	are continuous from the left and from the right, respectively.
	Given $\tau$, there is a  uniquely determined  $s\le0$ such that $\meas(\Omega_{<s}) \le \tau \le \meas(\Omega_{\le s})$.
	Since the measure space is non-atomic, the celebrated Sierpi\'nski theorem implies that there is $S\subset\Omega$ such that $\Omega_{< s} \subset S \subset \Omega_{\le s}$ and $\meas(S)=\tau$.

By the first part of the proof, $\tilde v$ is integrable on $S$.
Then $s$ and $S$ satisfy the requirements of \cref{lem_basic_ineq}.
Using \cref{lem_basic_ineq} and  \eqref{eq_tildev_lower}, we get
\[
 \int_\Omega  g(x,\bar u(x)) - g(x,0)\dx \ge \int_S \tilde v(x) \dx  \ge \int_\Omega g(x,\bar u(x)) - g(x,0)\dx.
\]
Hence, the inequality of \cref{lem_basic_ineq}  is satisfied with equality, which implies $ \Omega_{< s} \subset \supp \bar u \subset \Omega_{\le s}$.
It remains to utilize that $s=0$ or $\meas(\supp\bar u)=\tau$. If $\meas(\bar u)=\tau$ then \eqref{eq_optcon_support} and \eqref{eq_optcon_u_tildev}
are satisfied with $A:=\supp \bar u$.
If $\meas(\supp \bar u)<\tau$ then $s=0$, 
and we can find a set $A$ with $\meas(A)=\tau$ and $\supp \bar u \subset A \subset \Omega = \Omega_{\le 0}$, which is \eqref{eq_optcon_support}.
Using \cref{lem_basic_ineq} and $\bar u(x)=0$ on $A\setminus \supp\bar u$, we see that \eqref{eq_optcon_u_tildev} is satisfied.

Let now $\bar u,s,A$ satisfy \eqref{eq_optcon_support} and \eqref{eq_optcon_u_tildev} such that $\tilde v$ is integrable on $A$. Let $u\in \Utau$. Then by \cref{lem_basic_ineq} with $S:=A$
we find
\begin{multline*}
 \int_\Omega  g(x,u(x)) - g(x,0)\dx \ge \int_A \tilde v(x) \dx  =\int_A g(x,\bar u(x)) - g(x,0)\dx \\
 = \int_\Omega g(x,\bar u(x)) - g(x,0)\dx,
\end{multline*}
and $\bar u$ solves \eqref{eq_int_prob}.
\end{proof}

\begin{corollary}\label{cor_comp_vtilde}
Let \cref{ass_int_g} be satisfied.
 Let $\bar u$ be a solution of \eqref{eq_int_prob}.
  Let $s\le 0$ be given by \cref{lem_char_sol_tau}.
  Then for almost all $x\in \Omega$
  \[
  	|\bar u(x)|_0 \cdot ( \tilde v(x) - s) \le 0.
  \]
\end{corollary}
\begin{proof}
This follows from \cref{lem_char_sol_tau}, \eqref{eq_optcon_support}: If $\bar u(x)\ne 0$, then $\tilde v(x) \le s$.
\end{proof}

Let us define the value function of \eqref{eq_int_prob} by
\[
 V(\tau):= \inf_{u\in \Utau}  \int_\Omega g(x,u(x))\dx .
\]
Using the above characterization of solutions, we have the following strong stability result.

\begin{lemma}
Let \cref{ass_int_g} be satisfied.
Let $\tau,\tau' \in (0,\meas(\Omega))$ with $\tau < \tau'$ be given.
Then $0\le V(\tau)-V(\tau') \le |s| (\tau - \tau')$, where $s$ is associated to $\tau$ by \cref{lem_char_sol_tau}.
\end{lemma}
\begin{proof}
	Let $u_\tau, u_{\tau'}$ be solutions to $\tau,\tau'$. Due to \cref{lem_char_sol_tau} there are $s,s'$, $A,A'$ such that
	$\meas(A) = \tau$, $\meas(A') = \tau'$, and
	\[
	 \Omega_{< s} \subset A \subset \Omega_{\le s} , \quad \Omega_{< s'} \subset A' \subset \Omega_{\le s}.
	\]
	If $s<s'$ then $\Omega_{\le s} \subset  \Omega_{< s'}$ and $A\subset A'$, which implies
	\[\begin{split}
		\int_\Omega g(x, u_\tau(x))\dx - \int_\Omega g(x, u_{\tau'}(x))\dx
		&= \int_A \tilde v\dx - \int_{A'} \tilde v \dx \\
		&= - \int_{A'\setminus A} \tilde v \dx \\
		& \le - s \meas(A'\setminus A) = - s(\tau' - \tau).
	\end{split}\]
	If $s=s'$ then
	\begin{multline*}
		\int_\Omega g(x, u_\tau(x))\dx - \int_\Omega g(x, u_{\tau'}(x))\dx
		= \int_{A \setminus \Omega_{< s}}\tilde v\dx - \int_{A' \setminus \Omega_{< s}} \tilde v \dx \\
		= s (\meas(A \setminus \Omega_{< s}) -\meas(A' \setminus \Omega_{< s}))
		= - s(\tau' - \tau),
	\end{multline*}
resulting in the same estimate.
\end{proof}

In addition, we obtain the following result,
which says that
$-s$ can be interpreted as Lagrange multiplier to the constraint $\|u\|_0\le \tau$.

\begin{corollary}\label{cor25}
Let \cref{ass_int_g} be satisfied.
 Let $\bar u$ be a solution of \eqref{eq_int_prob}.
  Let $s\le 0$ be given by \cref{lem_char_sol_tau}.
Then we have
 \begin{equation}\label{eq_comp_cond}
  s \cdot ( \tau - \|\bar u\|_0) = 0.
 \end{equation}
\end{corollary}
\begin{proof}
Suppose $\|\bar u\|_0< \tau$.
By \cref{lem_char_sol_tau} there is $A$ with  $\meas(A)=\tau$ and
	 $ \supp \bar u \subset A\subset \Omega_{\le s}$.
Due to \eqref{eq_optcon_u_tildev}, $\tilde v=0$ on $A\setminus \supp \bar u \subset  \Omega_{\le s}$, where $A\setminus \supp \bar u$ has positive measure.
Hence, $s=0$ follows by definition of $\Omega_{\le s}$, see \eqref{eq_def_omegas}.
\end{proof}

Furthermore, $\bar u$ is a solution of unconstrained penalized problems,
where $-s$ plays the role of a penalization parameter.

\begin{corollary}\label{cor_int_penalty}
Let \cref{ass_int_g} be satisfied.
 Let $\bar u$ be a solution of \eqref{eq_int_prob}.
Let $\lambda:=-s\ge0$, where $s$ is given by \cref{lem_char_sol_tau}.
Then $\bar u$ is a solution of
\[
 \min_{u} \int_\Omega g(x,u(x)) + \lambda |u(x)|_0 \dx.
\]
and a solution of
\[
 \min_{u} \int_\Omega g(x,u(x)) \dx + \lambda ( \|u\|_0 - \tau)^+.
\]
\end{corollary}
\begin{proof}
Let $s$ and $A$ be as in \cref{lem_char_sol_tau}.
 Let $u$ be measurable and set $B:=\supp u$. As in the proof of  \cref{lem_basic_ineq}, we get
 \[
  \int_\Omega  g(x,u(x)) - g(x,0)\dx \ge \int_B \tilde v \dx \ge  \int_{A \cap B} \tilde v(x) \dx + s \meas(B \setminus A)
 \]
 and
 \[
  \int_{A \cap B} \tilde v(x) \dx + s \meas(A \setminus B)\ge \int_A \tilde v \dx =\int_\Omega g(x,\bar u(x)) - g(x,0)\dx,
 \]
 which results in
 \[
   \int_\Omega  g(x,u(x)) - g(x,\bar u(x)) \ge s (\meas(B \setminus A) -  \meas(A \setminus B)) .
 \]
We proceed with
\[\begin{split}
 s (\meas(B \setminus A) -  \meas(A \setminus B)) & = s(\meas(B \setminus A) - \meas(A \setminus B) +\|\bar u\|_0 - \|\bar u\|_0)\\
 &\ge s(\meas(B \setminus A) + \meas(A \cap B) - \|\bar u\|_0) \\
 &= s(\|u\|_0-\|\bar u\|_0),
\end{split}\]
where we used $\|\bar u\|_0 \le\meas(A)$, $\|u\|_0=\meas(B)$, and $s\le0$.
This proves the first claim.
 Using the result of  \cref{cor25} and $s\le0$, we get
\[
 s (\meas(B \setminus A) -  \meas(A \setminus B)) \ge s(\|u\|_0- \tau) \ge s (\|u\|_0- \tau )^+,
\]
which proves the second claim.
\end{proof}

Let us prove the following converse result.

\begin{corollary}
Let \cref{ass_int_g} be satisfied.
Let $\lambda'\ge0$.
 Let $\bar u$ with $\|\bar u\|_0=\tau$ be a solution of
\[
 \min_{v} \int_\Omega g(x,v(x)) + \lambda' |v(x)|_0 \dx.
\]
Then $\bar u$ solves \eqref{eq_int_prob}.
\end{corollary}
\begin{proof}
Let $u$ be given with $\|u\|_0\le \tau$.
By optimality of $\bar u$, we have
 \begin{multline*}
 \int_\Omega g(x,\bar u(x)) \dx+ \lambda' \tau
 =\int_\Omega g(x,\bar u(x)) + \lambda' |\bar u(x)|_0 \dx
 \\
 \le \int_\Omega g(x,u(x)) + \lambda' |u(x)|_0 \dx
 \le \int_\Omega g(x,u(x)) \dx + \lambda' \tau,
\end{multline*}
which implies the claim.
\end{proof}

Let us close the section with the following observation: Every minimum of the integral functional $\int_\Omega g(x,u(x))\dx$
subject to the constraint $u\in \Utau \cap L^p(\Omega)$
is a solution of  \eqref{eq_int_prob}.

\begin{theorem}\label{lem_g_nice}
Let \cref{ass_int_g} be satisfied.
 Let $p\in [1,\infty]$. Let $\bar u\in L^p(\Omega)$ be a solution of
 \[
  \min_{u \in \Utau \cap L^p(\Omega)} \int_\Omega g(x,u(x))\dx.
  \]
  Then $\bar u$ solves \eqref{eq_int_prob}.
\end{theorem}
\begin{proof}
Let $(u_n)$ and $(v_n)$ be given by \cref{lem_un_vn},
which implies  $u_n\in L^\infty(\Omega)$ for all $n$.
Let $B\subset \Omega$ with $\meas(B)\le \tau$ be given,
hence $\chi_B u_n\in L^p(\Omega)$ for all $n$.
Arguing as in the proof of \cref{lem_char_sol_tau}, we get
$\int_B v_n \dx\to \int_B \tilde v \dx  \ge  \int_\Omega g(x,\bar u(x)) - g(x,0)\dx$
by monotone convergence, see \eqref{eq_tildev_lower}.
Let now $u$ be feasible for \eqref{eq_int_prob}. Let $B:=\supp u$. Then
\[
 \int_\Omega g(x,u(x)) - g(x,0)\dx \ge \int_B \tilde v \dx \ge \int_\Omega g(x,\bar u(x))-g(x,0)\dx,
\]
hence $\bar u$ solves  \eqref{eq_int_prob} as well.
\end{proof}

\begin{remark}
All the results of this section are valid in the more general situation of a non-atomic, complete, $\sigma$-finite measure space.
\end{remark}

\section{Optimal control of elliptic partial differential equation with $L^0$ constraint}
\label{sec4}

In this section, we consider the following optimal control problem: Minimize
\begin{subequations}\label{eq_control_problem}
\begin{equation}\label{eq_nonlinear_fctal}
  \int_\Omega L(x,y_u(x), u(x))\dx
\end{equation}
subject to
\begin{equation}\label{eq_nonlinear_l0constr}
 \|u\|_0 \le \tau,
\end{equation}
where $y_u$ is the weak solution of the equation
\begin{equation}\label{eq_nonlinear_state}
\begin{aligned}
 (Ay)(x) &= f(x,y(x),u(x)) &&\text{ on } \Omega\\
 y&=0 &&\text{ on } \partial \Omega.
 \end{aligned}
\end{equation}
\end{subequations}
We impose the following assumption on the data of this problem:

\begin{assumption}\label{ass_elliptic}
\quad
 \begin{enumerate}
  \item $\Omega$ is an open and bounded domain in $\R^d$, $d \in \{2,3\}$, with Lipschitz boundary $\partial\Omega$.
  Let $\tau\in (0,\meas(\Omega))$.
  \item $A$ denotes a second-order elliptic operator in $\Omega$ of the type
\[Ay=-\sum_{i,j=1}^d\partial_{x_j}(a_{ij}(x)\partial_{x_i}y)\] with coefficients $a_{ij} \in C(\bar\Omega)$. In addition, there is $\Lambda>0$ such that
for almost all $x\in \Omega$
\[
\sum_{i,j=1}^d a_{ij}(x)\xi_i\xi_j \ge \Lambda |\xi|^2 \quad \forall \xi\in\mathbb{R}^d.
\]
\item The functions $f,L:\Omega \times \R \times \R$ are Caratheodory functions, i.e., $x\mapsto f(x,y,u)$ and $x\mapsto L(x,y,u)$ are measurable for all $y,u\in\R$,
and $(u,y) \mapsto f(x,y,u)$ and $(u,y) \mapsto  L(x,y,u)$ are continuous
for almost all $x\in \Omega$.
We assume that $f,L$ are continuously differentiable with respect to $y$ for almost all $x\in \Omega$ and all $u\in \R$
with $f_y(x,y,u)\le0$.
In addition, for all $M>0$ there are non-negative $a_M\in L^1(\Omega)$, $b_M\in \R$, $c_M\in L^2(\Omega)$ such that
for almost all $x\in \Omega$
\[
 |L(x,y,u)| +
 |L_y(x,y,u)| \le a_M(x) + b_M |u|^2 \quad \forall |y|\le M
\]
and
\[
|f(x,y,u)|+
 |f_y(x,y,u)| \le c_M(x) + b_M | u|  \quad \forall |y|\le M,
\]
where $f_y,L_y$ denote the partial derivatives of $f,L$ with respect to $y$.
 \end{enumerate}
\end{assumption}

Let  us briefly comment on those assumptions. The conditions on the differential equation are to ensure $W^{1,p}$ regularity of
weak solutions $y$ of \eqref{eq_nonlinear_state} for some $p>d$, which guarantees $y\in L^\infty(\Omega)$.
The conditions on $L$ and $f$ ensure that the Nemyzki operators induced by them are continuous (and differentiable with respect to $y$) from
$L^\infty(\Omega) \times L^2(\Omega)$ to $L^1(\Omega)$ and $L^2(\Omega)$, respectively.
We opted for this set of conditions in order to be able to use the results of \cite{Casas1994} on regularity of solutions of partial differential equations.
This allows us to focus on the $L^0$ constraints. Of course, other settings are possible (e.g., control constraints, other types of boundary conditions, parabolic equations).

As consequence of the assumptions, we have the following solvability and regularity result for \eqref{eq_nonlinear_state}.
\begin{theorem}
Let \cref{ass_elliptic} be satisfied. Let $u\in L^2(\Omega)$ be given. Then there is a uniquely determined $y_u\in W^{1,p}_0(\Omega)$ solving the equation \eqref{eq_nonlinear_state},
where $p>d$.
\end{theorem}
\begin{proof}
This is a consequence of \cite[Theorem 1]{Casas1994}.
\end{proof}

We define the Hamiltonian of the control problem \eqref{eq_control_problem} by
\[
	H(x,y,u,\phi):= L(x,y,u) + \phi f(x,y,u).
\]
Note that the inequality constraint $\|u\|_0\le \tau$ is not taken into account in the Hamiltonian,
which is different to the approach in \cref{sec2}.
In addition, we defined the Hamiltonian in the qualified sense, that is, there is no ``multiplier'' $\phi_0\ge0$ associated to the
functional $L$ by $\phi_0L$.

We are going to prove the maximum principle in integrated form first.
The main difference to other works, e.g., \cite{Casas1994,RaymondZidani1998}, is the construction of perturbations that satisfy the constraint $\|u\|_0\le \tau$.
Here, we will adapt a result of \cite{RaymondZidani1998} to generate these perturbations.
It is based on Lyapunov's theorem.

\begin{lemma}\label{lem_lyapunov}
Let $\rho\in (0,1)$.
 Let $g_1,\dots, g_m \in L^1(\Omega)$ be given. Then there is a sequence $(E_\rho^n)$ of measurable subsets of $\Omega$ such that
 \[
  \int_{E_\rho^n} g_k \dx = \rho\int_\Omega g_k\dx \quad \forall k=1,\dots, m\ \forall n\in \N
 \]
 and
 \[
  \frac1\rho \chi_{E_\rho^n} \rightharpoonup^* 1 \quad \text{ for } n\to\infty \text{ in } L^\infty(\Omega) = L^1(\Omega)^*.
 \]
\end{lemma}
\begin{proof}
 The proof is an adaptation of the proof of \cite[Lemma 4.2]{RaymondZidani1998}. It is included for the convenience of the reader.

 Let $(\phi_n)$ be a dense subset of $L^1(\Omega)$. For $n\ge0$, define $f^n:\Omega \to \R^{m+n}$ by
 \[
  f^n = (g_1,\dots, g_m, \phi_1,\dots, \phi_n).
 \]
 By the Lyapunov convexity theorem \cite[Corollary IX.5]{DiestelUhl1977}, there is $E_\rho^n \subset\Omega$ such that
 $\rho \int_\Omega f^n \dx = \int_{E_\rho^n} f^n \dx$. By definition of $f^n$, this implies $\int_{E_\rho^n} g_k \dx = \rho\int_\Omega g_k\dx$ for all $k$.

 Let now $\phi\in L^1(\Omega)$ be given. Take $\epsilon>0$. By density, there is $N$ such that $\|\phi- \phi_N\|_{L^1(\Omega)} <\epsilon$.
 Then for all $n>N$, we get
 \[\begin{split}
  \left| \int_\Omega (1- \frac1\rho\chi_{E_\rho^n}) \phi \dx \right| & \le \left| \int_\Omega (1- \frac1\rho\chi_{E_\rho^n}) (\phi-\phi_N) \dx \right|
  + \left| \int_\Omega (1- \frac1\rho\chi_{E_\rho^n}) \phi_N \dx \right|  \\
  &\le \frac{1-\rho}\rho \epsilon+ 0,
  \end{split}
 \]
which proves the claim.
\end{proof}

\begin{corollary}\label{cor_lyapunov}
 Let  $(E_\rho^n)$ be a sequence of measurable subsets of $\Omega$ such that
 \[
  \frac1\rho \chi_{E_\rho^n} \rightharpoonup^* 1 \quad \text{ for } n\to\infty \text{ in } L^\infty(\Omega) = L^1(\Omega)^*.
 \]
 Let $h \in L^2(\Omega)$ be given. Then $(1-\frac1\rho \chi_{E_\rho^n})h \to 0$ in $W^{-1,p}(\Omega)$ where
 $p\in (1,+\infty)$ for $d=2$ and $p\in (1,6)$ for $d=3$.
\end{corollary}
\begin{proof}
Due to the assumptions, we have $(1-\frac1\rho \chi_{E_\rho^n})h \rightharpoonup 0$ in $L^2(\Omega)$.
 Under the conditions on $p$, the embedding $W^{1,p'}_0(\Omega)\hookrightarrow L^2(\Omega)$ is compact,
 where $p'$ is given by $\frac1p+\frac1{p'}=1$.
 Hence, the embedding $L^2(\Omega)\hookrightarrow W^{-1,p}(\Omega)$ is compact as well.
\end{proof}

Now we have all tools available to prove the maximum principle.
The proof is very similar to the proofs in \cite{Casas1994}.
Hence, we will be brief on  arguments that are similar to those in  \cite{Casas1994}.
We first prove the maximum principle in integrated form.

\begin{theorem}\label{thm_pmp_integrated}
Let $\bar u$ be a local solution of \eqref{eq_nonlinear_fctal}--\eqref{eq_nonlinear_l0constr} in the $L^2(\Omega)$-sense
with associated state $\bar y:= y_{\bar u} \in W^{1,p}_0(\Omega)$,
where $p>d$ is such that $W^{1,p'}_0(\Omega)\hookrightarrow L^2(\Omega)$,
 where $p'$ is given by $\frac1p+\frac1{p'}=1$.
Then there is $\bar\phi\in W^{1,p'}_0(\Omega)$ that solves the adjoint equation
\[
 A^*\bar\phi  = f_y(\cdot,\bar y,\bar u)\bar\phi + L_y(\cdot,\bar y,\bar u).
\]
In addition,
\[
	\int_\Omega H(x,\bar y(x), \bar u(x), \bar \phi(x))\dx \le \int_\Omega H(x,\bar y(x), v(x), \bar \phi(x))\dx
\]
for all $v\in L^2(\Omega)$ be with $\|v\|_0 \le \tau$.
\end{theorem}
\begin{proof}
Let $v\in L^2(\Omega)$ with $\|v\|_0 \le \tau$.
Set $h:=f(\cdot, \bar y, v) - f(\cdot, \bar y, \bar u)$, $m:=4$, and
 \[
  (g_1, \dots, g_m):=((v-\bar u)^2,\, |\bar u|_0,\, |v|_0,\,  L(\cdot,\bar y,v)-L(\cdot,\bar y,\bar u)).
 \]
Then by \cref{lem_lyapunov} and \cref{cor_lyapunov}, for each $\rho>0$ there is a set $E_\rho$ such that
 \[
  \int_{E_\rho} g_k \dx = \rho\int_\Omega g_k\dx \quad \forall k=1,\dots, m
 \]
and
\[
 \left\| \left(1-\frac1\rho \chi_{E_\rho}\right)h \right\|_{W^{-1,p}(\Omega)} < \rho.
\]
 Let us set
 \[
  u_\rho = \bar u + \chi_{E_\rho}(v-\bar u).
 \]
Then
\[\begin{split}
 \|u_\rho\|_0 &= \|(1-\chi_{E_\rho})\bar u +  \chi_{E_\rho}v\|_0 \\& = \|(1-\chi_{E_\rho})\bar u \|_0+\|  \chi_{E_\rho}v\|_0\\
 &= \|\bar u \|_0 - \|\chi_{E_\rho}\bar u \|_0+\|  \chi_{E_\rho}v\|_0 = (1-\rho)\|\bar u \|_0 + \rho \|v\|_0 \le \tau
\end{split}\]
and
\[
 \|u_\rho -\bar u\|_{L^2(\Omega)}^2 = \int_{E_\rho}(v-\bar u)^2\dx = \rho \|v-\bar u\|_{L^2(\Omega)}^2 .
\]
Hence, $J(\bar u) \le J(u_\rho)$ by local optimality of $\bar u$ for $\rho>0$ small enough.
Arguing as in \cite[Lemma 2]{Casas1994}, we find
\[
 0\le \lim_{\rho\searrow0} \frac1\rho (J(u_\rho) - J(\bar u)) = z^0,
\]
where
\[
 z^0 = \int_\Omega L_y(x,\bar y(x),\bar u(x))z(x) + L(x,\bar y, v(x))-L(x,\bar y(x),\bar u(x)) \dx
\]
and $z\in W^{1,p}_0(\Omega)$ satisfies
\[
 Az  = f_y(\cdot,\bar y,\bar u)z + f(\cdot,\bar y, v)-f(\cdot,\bar y,\bar u).
\]
In addition, there is $\bar\phi\in W^{1,p'}_0(\Omega)$ \cite[Theorem 2]{Casas1994} that solves the adjoint equation
\[
 A^*\bar\phi  = f_y(\cdot,\bar y,\bar u)\bar\phi + L_y(\cdot,\bar y,\bar u).
\]
This implies
\[\begin{aligned}
 0\le z^0 & = \int_\Omega L(x,\bar y, v(x))-L(x,\bar y(x),\bar u(x)) \dx \\
 & \qquad +\int_\Omega \bar\phi(x) ( f(x,\bar y(x), v(x))-f(x,\bar y(x),\bar u(x))) \dx\\
 & = \int_\Omega H(x,\bar y(x), v(x), \bar \phi(x)) - H(x,\bar y(x), \bar u(x), \bar \phi(x))\dx,
\end{aligned}\]
which is the claim.
\end{proof}

Using \cref{lem_g_nice} and the results of \cref{sec3}, we can turn the maximum principle from integrated to pointwise form.

\begin{theorem}\label{thm_pmp_pointwise}
Let $\bar u$ be a local solution of \eqref{eq_nonlinear_fctal}--\eqref{eq_nonlinear_l0constr} in the $L^2(\Omega)$-sense
with associated state $\bar y:= y_{\bar u} \in W^{1,p}_0(\Omega)$,
where $p>d$ is such that $W^{1,p'}_0(\Omega)\hookrightarrow L^2(\Omega)$
 where $p'$ is given by $\frac1p+\frac1{p'}=1$, and
adjoint $\bar\phi\in W^{1,p'}_0(\Omega)$ given by \cref{thm_pmp_integrated}.

Then there is a number $s\le0$ such that
\[
	s(\|\bar u\|_0-\tau)=0
\]
and for almost all $x\in \Omega$
\[
\bar u(x) = \argmin_{ u\in \R}  H(x,\bar y(x), u, \bar \phi(x))+ (-s) |u|_0 .
\]
In addition, we have the following properties for almost all $x\in \Omega$:
\[
\bar u(x)\ne0 \quad \Rightarrow \quad \bar u(x) = \argmin_{ u\in \R}  H(x,\bar y(x), u, \bar \phi(x)),
\]
\[
|\bar u(x)|_0 \cdot (\inf_{u\in \R}H(x,\bar y(x), u, \bar \phi(x))  - H(x,\bar y(x), 0, \bar \phi(x)) -s )\le 0.
\]
\end{theorem}
\begin{proof}
Let us define $g$ by
\[
	g(x,u):= H(x,\bar y(x), u, \bar \phi(x)).
\]
Then $g$ is a normal integrand, and $g(\cdot,0)$ is integrable. Due to \cref{thm_pmp_integrated},
$\bar u$ solves
\[
	\min_{u\in \Utau \cap L^2(\Omega)} \int_\Omega g(x,u(x))\dx.
\]
By \cref{lem_g_nice}, $\bar u$ solves \eqref{eq_int_prob}.
Hence, the results of \cref{sec3} are applicable.
Let $s\le0$ be as in \cref{lem_char_sol_tau}.
Then the claim follows with \cref{cor_comp_vtilde,cor25,cor_int_penalty}.
%
\end{proof}

This result shows that the conditions \eqref{eq_ode_pmp1} and  \eqref{eq_ode_pmp2}, which we derived for
an ODE control problem, are satisfied in adapted form in the PDE control problem.

\section{Proximal gradient algorithm}\label{sec5}

In this section, we will analyze a proximal gradient algorithm applied to a problem with $L^0$ constraints.
Here, we consider problems of the type
\begin{equation}\label{eq_prob_sec5}
	\min_{u\in \Utau \cap L^2(\Omega)} f(u) + \frac\alpha2 \|u\|_{L^2(\Omega)}^2 .
\end{equation}

We are going to use the following set of assumptions.

\begin{assumption}\label{ass_proxgrad}
\phantom{bla bla}
\begin{enumerate}
\item $\Omega\subset\R^d$ is Lebesgue measurable with $\meas(\Omega)\in (0,\infty)$, $\tau \in (0,\meas(\Omega))$.
 \item\label{ass_a2}
The function  $f:L^2(\Omega)\to \R$ is bounded from below and Fréchet differentiable.
In addition, $\nabla f:L^2(\Omega)\to L^2(\Omega)$ is Lipschitz continuous with constant $L_f$, i.e.,
\[
\|\nabla f(u_1)-\nabla f(u_2)\|_{L^2(\Omega)}\leq L_f\|u_1-u_2\|_{L^2(\Omega)}
\]
holds for all  $u_1,u_2\in L^2(\Omega)$.

\item $\alpha\ge0$.

\end{enumerate}
\end{assumption}

These requirements on $f$ are well-established in the context of first-order optimization methods.
The requirement of global Lipschitz continuity of $\nabla f$ and knowledge of the Lipschitz modulus $L_f$ can be overcome
by a suitable back-tracking method, see \cite[Section 3.3]{DWachsmuth2019}, which can be used in our situation as well.

\begin{remark}\label{rem_pde_problem_fits}
Under some restrictions, the problem of \cref{sec4} satisfies these assumptions.
Let us assume that $L$ is of the form $L(x,y,u)=L(x,y) + \frac\alpha2 u^2$.
Define $f(u):= \int_\Omega L(x,y_u(x))\dx$, where $y_u$ is the solution of \eqref{eq_nonlinear_state}.
If the nonlinearity in the equation is linear in $u$, e.g., $f(x,y,u) = f(x,y)+u$, then $f$ satisfies \cref{ass_proxgrad}.
See also the discussion in \cite[Section 2.2]{NatemeyerWachsmuth2021}.
\end{remark}

Let us first prove a necessary optimality condition for \eqref{eq_prob_sec5}. The proof is similar to \cref{thm_pmp_integrated} above.

\begin{theorem}\label{thm_noc_sec5}
Suppose $f$ is a Fréchet differentiable mapping from $L^1(\Omega)\to \R$.
Let $\bar u$ be a local solution of \eqref{eq_prob_sec5}.
Then it holds
\[
	\int_\Omega \nabla f(\bar u) \bar u \dx +\frac\alpha2 \|\bar u\|_{L^2(\Omega)}^2\le \int_\Omega f(\bar u) v \dx + \frac\alpha2 \|v\|_{L^2(\Omega)}^2
\]
for all $v\in L^2(\Omega)$ be with $\|v\|_0 \le \tau$.

In addition, there is a number $s\le0$ such that
\begin{equation}\label{eq_noc5_comp_tau}
	s (\|\bar u\|_0-\tau) =0.
\end{equation}
If $\alpha>0$ then for almost all $x\in \Omega$ the following conditions are fulfilled:
\begin{equation}\label{eq_noc5_nonzero_u}
\bar u(x)\ne0 \quad \Rightarrow \quad \bar u(x) = -\frac1\alpha \nabla f(\bar u)(x),
\end{equation}
\begin{equation}\label{eq_noc5_comp_u}
|\bar u(x)|_0 \cdot ( -\frac1{2\alpha} |\nabla f(\bar u)(x)|^2 - s) \le 0.
\end{equation}
If $\alpha =0$ then $\nabla f(\bar u)=0$.
\end{theorem}
\begin{proof}
Let us set $F(u):=f(u) + \frac\alpha2 \|u\|_{L^2(\Omega)}^2$, which is Fréchet differentiable on $L^2(\Omega)$ with gradient $\nabla F(u)=\nabla f(u)+\alpha u$.
Let $v\in L^2(\Omega)$ with $\|v\|_0 \le \tau$.
Set $m:=5$, and
 \[
  (g_1, \dots, g_m):=((v-\bar u)^2,\, |\bar u|_0,\, |v|_0,\,   \nabla F(\bar u), \ |v-\bar u|).
 \]
Then by \cref{lem_lyapunov}, for each $\rho>0$ there is a set $E_\rho$ such that $\int_{E_\rho} g_j \dx = \rho\int_\Omega g_j\dx$ for all $j=1\dots m$.
As in the proof of \cref{thm_pmp_integrated}, the function
$u_\rho := \bar u + \chi_{E_\rho}(v-\bar u)$ satisfies
$ \|u_\rho\|_0 \le \tau$
and
$ \|u_\rho -\bar u\|_{L^2(\Omega)}^2 =  \rho \|v-\bar u\|_{L^2(\Omega)}^2$.
Due to Fréchet differentiability and the construction of $E_\rho$ and $u_\rho$, we have
\[\begin{split}
 F(u_\rho) - F(\bar u) &= \nabla F(\bar u)(u_\rho-\bar u) + o(\|u_\rho -\bar u\|_{L^1(\Omega)}) + \frac\alpha2 \|u_\rho -\bar u\|_{L^2(\Omega)}^2\\
 &= \rho  \nabla F(\bar u)(v-\bar u) + o(\rho) + \rho \frac\alpha2 \|v -\bar u\|_{L^2(\Omega)}^2.
 \end{split}
\]
Dividing by $\rho>0$ and
passing to the limit $\rho\searrow0$, implies by local optimality
\[\begin{split}
	0&\le \nabla F(\bar u)(v-\bar u) +  \frac\alpha2 \|v -\bar u\|_{L^2(\Omega)}^2 \\
	&= \nabla f(\bar u)(v-\bar u) + \frac\alpha2 \|v\|_{L^2(\Omega)}^2 - \frac\alpha2 \|\bar u\|_{L^2(\Omega)}^2,
\end{split}\]
which proves the first claim. The second claim follows from \cref{lem_char_sol_tau} and \cref{cor_comp_vtilde,cor25}
with
$\tilde v = -\frac1{2\alpha} |\nabla f(\bar u)|^2$.
\end{proof}

Let us briefly give the motivation of the proximal gradient algorithm.
The well-known steepest descent method applied to the unconstrained differentiable problem $\min_u f(u)$
amounts to the iteration
\begin{equation}\label{eq_steepest_d}
	u_{k+1} = u_k - t_k \nabla f(u_k),
\end{equation}
where $t_k>0$ is a suitable step-size. It is immediate that $u_{k+1}$ is a solution of the unconstrained problem
\begin{equation}\label{eq_proxgrad_motiv}
	\min_u f(u_k)+\nabla f(u_k)\cdot(u-u_k) + \frac{1}{2t_k}\|u-u_k\|^2_{L^2(\Omega)}.
\end{equation}
While it is impossible to add the constraint $\|u\|_0\le \tau$ to the iteration procedure \eqref{eq_steepest_d},
this constraint can be easily imposed on the problem \eqref{eq_proxgrad_motiv}.
The resulting proximal gradient (or forward-backward) algorithm  reads as follows.
Here, we replaced the parameter $t_k$ by a fixed parameter $L$, which takes the place of $\frac1{t_k}$.

\begin{algorithm}[Proximal gradient algorithm] \label{alg_proxgrad} Choose $L>0$ and $u_0\in L^2(\Omega)$. Set $k=0$.
	\begin{enumerate}
		\item Compute $u_{k+1}$ as solution of
		\begin{equation}\label{eq_proxgrad_sub}
			\min_{u\in \Utau \cap L^2(\Omega)} f(u_k)+\nabla f(u_k)\cdot(u-u_k) + \frac{L}{2}\|u-u_k\|^2_{L^2(\Omega)} + \frac\alpha2 \|u\|_{L^2(\Omega)}^2
		\end{equation}
		\item Set $k:=k+1$, go to step 1.
\end{enumerate} \end{algorithm}

The functional to be minimized in \eqref{eq_proxgrad_sub} can be written as an integral functional $\int_\Omega g(x,u(x))\dx$ with
$g$ defined by
\[
g(x,u) =  f(u_k) + \nabla f(u_k)(x)\cdot(u-u_k(x)) + \frac{L}{2}(u-u_k(x))^2 + \frac\alpha2 u^2.
\]
The pointwise minimum of $g$ is realized by the function $\tilde u\in L^2(\Omega)$ defined by
\[
	\tilde u(x):= \frac{Lu_k(x) - \nabla f(u_k)(x) }{L+\alpha}.
\]
Clearly, $g(\cdot,0)$ is integrable, \cref{ass_int_g} is satisfied, and
the results of \cref{sec3} are applicable.
Hence, a solution of \eqref{eq_proxgrad_sub} can be computed as in \cref{lem_char_sol_tau}.
Here, $L>0$ is important: note that integrability of $g(\cdot,u)$ implies $u\in L^2(\Omega)$.
It is easy to verify that
\[
	 g(x,\tilde u(x))-g(x,0) = -\frac1{2(L+\alpha)} \left(Lu_k(x) - \nabla f(u_k)(x) \right)^2,
\]
which corresponds to $\tilde v$ in \cref{lem_char_sol_tau}.

\begin{lemma}\label{lem_sol_proxgrad}
Let \cref{ass_proxgrad} be satisfied.
	Let $L>0$ and $u_k\in L^2(\Omega)$ be given.
	Then \eqref{eq_proxgrad_sub} is solvable. In addition, there is $\lambda_{k+1}\ge0$ such that for every solution $u_{k+1}$ of \eqref{eq_proxgrad_sub} it holds
	\[
		\lambda_{k+1}( \|u_{k+1}\|_0- \tau) =0,
	\]
    and $u_{k+1}$ solves
    \[
    	\min_{u\in L^2(\Omega)} f(u_k)+\nabla f(u_k)\cdot(u-u_k) + \frac{L}{2}\|u-u_k\|^2_{L^2(\Omega)} + \frac\alpha2 \|u\|_{L^2(\Omega)}^2 + \lambda_{k+1}\|u\|_0.
    \]
    Moreover, for almost all $x\in\Omega$ we have
    \begin{equation}\label{eq_iterates_away_zero}
    	u_{k+1}(x) \ne 0 \quad \Rightarrow \quad |u_{k+1}(x)| \ge \sqrt{ \frac{2 \lambda_{k+1}}{L+\alpha} }
    \end{equation}
    and
    \begin{equation}\label{eq_comp_vtilde_iterates}
    	|u_{k+1}(x)|_0 \cdot \left(-\frac1{2(L+\alpha)} \left(Lu_k(x) - \nabla f(u_k)(x) \right)^2  + \lambda_{k+1}\right) \le 0.
    \end{equation}
\end{lemma}
\begin{proof}
	Existence of solutions follows from \cref{lem_char_sol_tau}.
	The properties of $\lambda_{k+1}:=-s$, where $s$ is as in  \cref{lem_char_sol_tau}, are consequences of \cref{cor25} and \cref{cor_int_penalty}.
	The claim \eqref{eq_iterates_away_zero} is a consequence of \cref{cor_int_penalty} and \cite[Corollary 3.9]{DWachsmuth2019}.
	Finally, \cref{cor_comp_vtilde} implies \eqref{eq_comp_vtilde_iterates}.
\end{proof}

The iterates of the algorithm satisfy the following properties.

\begin{theorem}\label{thm_proxgrad_basic}
Let \cref{ass_proxgrad} be satisfied. Suppose $L>L_f$.
 Let $(u_k)$ be a sequence of iterates generated by \cref{alg_proxgrad}.
 Then it holds that:
 \begin{enumerate}
  \item The sequences $(u_k)$ and $(\nabla f(u_k))$ are bounded in $L^2(\Omega)$ if $\alpha>0$.
  \item The sequence $(f(u_k) + \frac\alpha2\|u_k\|_{L^2(\Omega)})$ is monotonically decreasing and converging.
  \item $\sum_{k=0}^\infty \|u_{k+1}-u_k\|_{L^2(\Omega)}^2 <\infty$.
 \end{enumerate}
\end{theorem}
\begin{proof}
These claims can be proven as in \cite[Theorem 3.13]{DWachsmuth2019}.
\end{proof}

Let us define the following sequence
\[
	\chi_k(x):= |u_k(x)|_0.
\]
Using \cref{eq_iterates_away_zero}, we have the following estimate of $(\chi_k)$,
which is similar to \cite[Lemma 3.12]{DWachsmuth2019}.

\begin{lemma}
Let $(u_k)$ be iterates of \cref{alg_proxgrad}. Then it holds
\[
	\|u_{k+1}-u_k\|_{L^2(\Omega)}^2 \ge  \frac{2 \min(\lambda_k,\lambda_{k+1})}{L+\alpha}  \|\chi_{k+1}-\chi_k\|_{L^1(\Omega)}.
\]
\end{lemma}
\begin{proof}
Let $x\in \Omega$ such that $\chi_{k+1}(x)\ne \chi_k(x)$.
Then $|\chi_{k+1}(x)- \chi_k(x)|=1$, and exactly one of $u_{k+1}(x)$ and $u_k(x)$ is zero.
Suppose $u_{k+1}(x)=0$ and $u_k(x)\ne0$. Then $|u_{k+1}(x)-u_k(x)|=|u_k(x)|\ge \sqrt{ \frac{2 \lambda_k}{L+\alpha} }$.
If $u_{k+1}(x)\ne0$ and $u_k(x)=0$ then $|u_{k+1}(x)|\ge \sqrt{ \frac{2 \lambda_{k+1}}{L+\alpha} }$.
And the claim is proven.
\end{proof}

Under the assumption that $(\lambda_k)$ is bounded from below by a positive number, we can
prove feasibility of weak limit points of the algorithm.
In the general situation, it is not clear how to prove such a result as the map $u\mapsto \|u\|_0$ is not
weakly sequentially lower semi-continuous from $L^2(\Omega) \to \R$.

\begin{theorem}
Let \cref{ass_proxgrad} be satisfied. Suppose $L>L_f$.
 Let $(u_k)$ be a sequence of iterates generated by \cref{alg_proxgrad}.
Suppose
\[
	\liminf_{k\to\infty} \lambda_k >0,
\]
where $(\lambda_k)$ is as in \cref{lem_sol_proxgrad}.

Then $\chi_k\to \bar\chi$ in $L^1(\Omega)$, and
every weak sequential limit point $\bar u$ of $(u_k)$ is feasible for the $L^0$ constraint, i.e. $\|\bar u\|_0\le \tau$.
\end{theorem}
\begin{proof}
	Let $\lambda:=\liminf_{k\to\infty} \lambda_k>0$. Then for all $k$ sufficiently large, we have
\[
	\|u_{k+1}-u_k\|_{L^2(\Omega)}^2 \ge  \frac{\lambda}{L+\alpha}  \|\chi_{k+1}-\chi_k\|_{L^1(\Omega)}.
\]
The summability of $\|u_{k+1}-u_k\|_{L^2(\Omega)}^2$ implies those of $\|\chi_{k+1}-\chi_k\|_{L^1(\Omega)}$.
Hence $(\chi_k)$ is a Cauchy sequence in $L^1(\Omega)$, $\chi_k\to \bar\chi$ in $L^1(\Omega)$, and $\bar\chi$ is a characteristic function.
As $(\chi_k)$ is trivially bounded in $L^\infty(\Omega)$, it follows $\chi_k\to \bar\chi$ in $L^p(\Omega)$ for all $p<\infty$.

Let now $(u_{k_n})$ be a subsequence with $u_{k_n} \rightharpoonup \bar u$ in $L^2(\Omega)$.
Let $\phi\in L^\infty(\Omega)$. Since $\chi_k(x):= |u_k(x)|_0$, we have $\int_\Omega (1-\chi_{k_n})u_{k_n} \phi\dx=0$ for all $n$.
Passing to the limit in this
equation yields $\int_\Omega (1-\bar \chi) \bar u \phi\dx=0$. Since $\phi\in L^\infty(\Omega)$ was arbitrary, this implies $ (1-\bar \chi) \bar u=0$ almost everywhere,
which in turn implies $|\bar u|_0 \le \bar\chi$ almost everywhere,
as both functions $\bar\chi$ and $|u|_0$ only attain the values $0$ and $1$.
And it follows
\[
	\|\bar u\|_0 \le \|\bar \chi\|_{L^1(\Omega)} = \lim_{k\to\infty} \|\chi_k\|_{L^1(\Omega)}  = \lim_{k\to\infty} \|u_k\|_0 \le \tau,
\]
and $\bar u$ is feasible for the $L^0$ constraint.
\end{proof}

Moreover, we can prove strong convergence under additional assumptions on $\nabla f$.
See also the related result \cite[Theorem 3.18]{DWachsmuth2019}.
Here, we assume that $\nabla f$ maps weakly to strongly converging sequences.

\begin{theorem}\label{thm_strong_conv_iterates}
Let \cref{ass_proxgrad} be satisfied. Suppose $L>L_f$.
Let us assume complete continuity of $\nabla f$ from $L^2(\Omega)$ to $L^2(\Omega)$,
i.e., for all sequences $(v_k)$ in $L^2(\Omega)$ the following implication
\be\label{eqcompletecont}
 v_k \rightharpoonup v \text{ in } L^2(\Omega) \ \Rightarrow \ \nabla f(v_k) \to \nabla f(v)  \text{ in } L^2(\Omega)
\ee
holds.
In addition, we require $\alpha>0$.

Let $(u_k)$ be a sequence of iterates generated by \cref{alg_proxgrad}.
Suppose
\begin{equation}\label{eq_ass_lambdak}
	\liminf_{k\to\infty} \lambda_k >0,
\end{equation}
where $(\lambda_k)$ is as in \cref{lem_sol_proxgrad}.

Then $u_{k_n}\rightharpoonup \bar u$ in $L^2(\Omega)$ implies $u_{k_n}\to \bar u$ in $L^2(\Omega)$.
In addition, for almost all $x\in \Omega$ the following condition is fulfilled
\[
\bar u(x)\ne0 \quad \Rightarrow \quad \bar u(x) = -\frac1\alpha \nabla f(\bar u)(x).
\]
\end{theorem}
\begin{proof}
If $u_{k+1}(x)\ne0$ then $\alpha u_{k+1}(x) = - (\nabla f(u_k)(x) + L(u_{k+1}(x)-u_k(x)))$. This implies
\[
\chi_{k+1} u_{k+1} =- \chi_{k+1} \frac1\alpha (\nabla f(u_k) + L(u_{k+1}-u_k)).
\]
Adding the equation $(1-\chi_{k+1}) u_{k+1} =0$, yields
\begin{equation}\label{eq510}
	\alpha u_{k+1} =- \chi_{k+1} (\nabla f(u_k) + L(u_{k+1}-u_k)).
\end{equation}
Let now $u_{k_n}\rightharpoonup \bar u$ in $L^2(\Omega)$. Then $\nabla f(u_{k_n})\to \nabla f(\bar u)$ in $L^2(\Omega)$ by complete continuity of $\nabla f$.
In addition, $u_{k+1} - u_k\to 0 $ in $L^2(\Omega)$ by \cref{thm_proxgrad_basic}.
The right-hand side in \eqref{eq510} converges strongly in $L^2(\Omega)$ by \cref{lem_lift_convergence} below,
which implies the strong convergence $u_{k_n}\to \bar u$ in $L^2(\Omega)$.
In addition, in the limit we obtain from \eqref{eq510}
$\alpha\bar u =- \bar\chi \nabla f(\bar u)$.
\end{proof}

Let us compare the properties of limit points $\bar u$ with the necessary optimality conditions \eqref{eq_noc5_comp_tau}--\eqref{eq_noc5_comp_u} according to \cref{thm_noc_sec5}.
The above result only proves the implication \eqref{eq_noc5_nonzero_u} for limit points $\bar u$.
It seems to be impossible to prove the remaining two conditions \eqref{eq_noc5_comp_tau} and \eqref{eq_noc5_comp_u}.
An obvious choice for $s$ in those formulas would be any limit point of $(-\lambda_k)$.
Under the assumption \eqref{eq_ass_lambdak}, we would get $s<0$. However, it seems impossible to prove $\|\bar u\|_0=\tau$:
the mapping $u\mapsto |u|_0$ is merely lower semicontinuous at $u=0$, so that we can only prove $\|\bar u\|_0 \le \tau$.
And it is not clear that the complementarity condition \eqref{eq_noc5_comp_tau} is satisfied in the limit.
In order to prove \eqref{eq_noc5_comp_u}, a natural idea would be to pass to the limit in the condition \eqref{eq_comp_vtilde_iterates}.
At best we can expect to get
\[
  |\bar u(x)|_0 \cdot \left(-\frac1{2(L+\alpha)} \left(L\bar u(x) - \nabla f(\bar u)(x) \right)^2  -s\right) \le 0.
\]
This is different to \eqref{eq_noc5_comp_u} because of the presence of the prox-parameter $L>0$ in the inequality.

We close the section with the following auxiliary result, whis was used in the proof of \cref{thm_strong_conv_iterates}. Note that an application of H\"older inequality
implies strong convergence in $L^p(\Omega)$ only for $p<2$.

\begin{lemma}\label{lem_lift_convergence}
Let $\meas(\Omega)<\infty$.
Let sequences $(\chi_k)$ and $(g_k)$ be given such that $\|\chi_k\|_{L^\infty(\Omega)}\le1$, $\chi_k \to \bar \chi$ in $L^1(\Omega)$, and $g_k\to \bar g$ in $L^2(\Omega)$.
Then $\chi_k g_k \to \bar \chi \bar g$ in $L^2(\Omega)$.
\end{lemma}
\begin{proof}
The sequences admit pointwise a.e.\@ converging subsequences $(g_{k_n})$, $(\chi_{k_n})$ together with a dominating function $a\in L^2(\Omega)$ with $|g_{k_n}|\le a$,
see \cite[Theorem 4.9]{Brezis11}.
Then $\chi_{k_n} g_{k_n} \to \bar \chi \bar g$ in $L^2(\Omega)$ by dominated convergence.
A subsequence-subsequence argument finishes the proof.
\end{proof}

\section*{Acknowledgement}

The author wants to thank an anonymous referee for pointing out a serious error in the submitted manuscript.

\section*{Data availability statement}

Data sharing not applicable to this article as no datasets were generated or analysed during the current study.

\printbibliography

\end{document}